\documentclass[oneside,reqno]{amsart}

\usepackage{graphicx}

\usepackage{subfig}
\usepackage{amsmath}
\usepackage{mathrsfs} 
\usepackage{amsfonts}
\usepackage{amssymb}%
\usepackage{amsthm}
\usepackage{amscd}
\usepackage{verbatim}
 \usepackage{mathabx}
 \usepackage{dsfont}
\usepackage{endnotes}
\usepackage{enumerate}
\usepackage{mathtools}
\usepackage{comment}
\usepackage[latin1]{inputenc}
\usepackage{tikz}
\usepackage{color}
\usepackage{hyperref}

\newcommand{\mcN}{\mathcal{N}}
\newcommand{\mcC}{\mathcal{C}}
\newcommand{\mcL}{\mathcal{L}}

\newcommand{\mcP}{\mathcal{P}}
\newcommand{\mcI}{\mathcal{I}}

\newcommand{\mcR}{\mathcal{R}}

\newcommand{\mcS}{\mathcal{S}}

\newcommand{\C}{\mathbb{C}}
\newcommand{\R}{\mathbb{R}}
\newcommand{\N}{\mathbb{N}}

\newcommand{\ZZ}{\mathbb{Z}}

\newcommand{\ot}{\otimes}

\newcommand{\GL}{\operatorname{GL}}

\newcommand{\tn}{\textnormal}

\newcommand{\End}{\textnormal{End}}
\newcommand{\ott}{\bigotimes}

\newcommand{\bop}{\bigoplus}
\newcommand{\Tr}{\textnormal{Tr}}

\newcommand{\actson}{\curvearrowright}

\theoremstyle{plain}
\newtheorem{theorem}{Theorem}[section]
\newtheorem{corollary}[theorem]{Corollary}
\newtheorem{proposition}[theorem]{Proposition}
\newtheorem{lemma}[theorem]{Lemma}

\newtheorem{observation}[theorem]{Observation}

\newtheorem{problem}[theorem]{Problem}
\theoremstyle{definition}
\newtheorem{defn}[theorem]{Definition}
\theoremstyle{definition}

\theoremstyle{definition}
\newtheorem{example}[theorem]{Example}
\theoremstyle{definition}

\theoremstyle{definition}

\title{On Subtilings of Polyomino Tilings}
\author{Jacob Turner${}^\dagger$}
\thanks{${}^\dagger$ Korteweg-de Vries Institute for Mathematics, University of Amsterdam, 1098 XG Amsterdam, Netherlands.}

\begin{document}

 \begin{abstract}
 We consider a problem concerning tilings of rectangular regions by a finite library of polyominoes. We specifically look at rectangular regions of dimension $n\times m$ and ask whether or not a tiling of this region can be rearranged so that tiling of the $n\times m$ rectangle can be realized as a tiling of an $n\times m'$ rectangle and an $n\times m''$ rectangle, $m=m'+m''$. We call this a subtiling. We show that the associated decision problem is $\mathsf{NP}$-complete when restricted to rectangular polyominoes. We also show that for certain finite libraries of polyominoes, if $m$ is sufficiently large, a subtiling always exists and give bounds.
 \end{abstract}
  \maketitle
  \begin{center}{\small {\sc Keywords:} Polyomino; tilings; algebraic combinatorics}
\end{center}
\section{Introduction}

We begin by considering the following combinatorial optimization problem. Suppose one has a set of objects that one can place in some rectangular container without having any gaps. However, the container may be considered suboptimal as one of its dimensions is very large. Instead, one would rather place the objects into two containers, both of which have more reasonable dimensions. This can be done of course, but at the expense of packing efficiency; one may be forced to accept that there will be wasted space in at least one of the containers.

 We model this problem in terms of polyomino tilings. Suppose you are given a finite set of shapes $\mcL$, which we will take to be polyominoes. Recall that a polyomino is a region of $\R^2$, defined up to translation, consisting of squares of equal size arranged by coincident edges. In this paper, we take the squares making up a polyomino to all be the unit square.
 
 We say that a set of polyominoes $\mcL$ can tile a compact region $R\subseteq \R^2$ if there is a multiset with elements drawn from $\mcL$ such that they can be arranged in $R$ so the every point in $R$ is covered by an element of the multiset and no overlap of two different elements is two dimensional.
 
 We are interested in tiling rectangular regions of $\R^2$ by a set of polyominoes $\mcL$. We denote an $n\times m$ rectangular region by $R_{n\times m}$. Suppose we having a tiling of $R_{n\times m}$ by a multiset $T$ of polyominoes with elements drawn from $\mcL$.  In general there may be several ways that the elements of $T$ tile $R_{n\times m}$ and it might be possible do so by tiling the regions $R_{n\times m'}$ and $R_{n\times m''}$ where $m=m'+m''$, with  $m',m''\ne 0$, using the elements from the multiset $T$ situtating these two tiled regions side by side. If this happens, we say that tiling of $R_{n\times m}$ admits a \emph{subtiling}. If a subtiling exists once one or more of the polyominoes in $T$ has been replaced with a rotated copy of itself, we say the tiling admits a subtiling \emph{allowing rotations}.  We make the following definition:
 
 \begin{defn}
  Suppose one is given a tiling $\mathcal{T}$ of $R_{n\times m}$ by a multiset of polyominoes $T$. Any tiling of $R_{n\times m}$ by the same multiset $T$ is called a \emph{rearrangement} of $\mathcal{T}$. We consider $\mathcal{T}$ to be a rearrangement of itself. A tiling of $R_{n\times m}$ by a multiset $T'$, reached by replacing any of the elements in $T$ by a rotated copy, is called a rearrangement of $\mathcal{T}$ \emph{with rotations}. We emphasize that while $T$ may contain multiple copies of a particular polyomino, each copy can be rotated different degrees independent of each other.
 \end{defn}

  An equivalent way of defining a subtiling is that a tiling $\mathcal{T}$ of $R_{n\times m}$ can be rearranged to give a tiling of $R_{n\times m}$ that is formed by juxtaposing tilings of $R_{n\times m'}$ and $R_{n\times m''}$ with $m=m'+m''$ and $m',m''\ne 0$, side by side in one of the two obvious ways. The analagous statement holds for subtilings allowing rotations. This motivates the following definition.
 
 \begin{defn}
  Given a finite set of polynominoes, define $\beta_n(\mcL)$ to be the smallest natural number such that for any tiling of $R_{n\times m}$ by elements of $\mcL$ with $m>\beta_n(\mcL)$, the tiling admits a subtiling. If arbitrarily large tilings exist with no subtilings, then $\beta_n(\mcL):=\infty$. We define $\beta^R_n(\mcL)$ analagously for subtilings allowing rotations. Note that these numbers depend on $n$. 
 \end{defn}

The main problem of interest for us is finding upper bounds on $\beta_n(\mcL)$. This forces us to address a couple of other questions that we will attempt to give at least partial answers to.
\begin{itemize}
 \item How hard is it to tell if a given tiling admits a subtiling, with or without allowing rotations?
 \item For which $\mcL$ and which $n$ is $\beta_n(\mcL)$ and $\beta_n^R(\mcL)$ finite?
\end{itemize}

We can answer the first question for finite sets of rectangular polyominoes and investigation into the latter question will occupy much of this paper. In this paper, we find several sets of polyominoes $\mcL$ such that $\beta_n(\mcL)$ is finite and give bounds on it. We will not devote much space to $\beta_n^R(\mcL)$ in this paper, but we note that $\beta_n(\mcL)\ge\beta_n^R(\mcL)$. 

Since the popularization of polyominoes in the mid 20th century, there has been a vast amount of literature on the subject and subsequent generalizations to polyforms, especially the question of determining when polyforms tile given regions \cite{golomb1989polyominoes,de1969filling,klarner1969packing,barnes1982algebraic,barnes1982algebraic2,marshall1997packing,conway1990tiling}.

Many polyomino problems have been recreational: for example puzzles of Martin Gardner \cite{gardner1970mathematical,gardner1975more,gardner1969unexpected} or more recently, the Eternity puzzle \cite{peggeternity,selbyeternity}; a survey of the recreational aspects of the subject can be found in \cite{martin1991polyominoes}. Others are problems well known as favorites to give young mathematics students to solve on homeworks, e.g. the problems of tiling mutilated chess boards by dominos. Some of these problems can be quite difficult. 

In addition, polyomino tiling problems have played an important part in complexity theory where they have proven to be a source of $\mathsf{NP}$-hard or even undecidable problems \cite{berger1966undecidability,robinson1971undecidability,beauquier1995tiling,moore2001hard}. Furthermore, as polyomino tilings are ubiquitous in combinatorial optimization and operations research, they have even found their way into other scientific and industrial settings with applications ranging from genetics to antenna design \cite{takefuji2012neural,chambers1998practical,karademir2015irregular,chen2004error}.

The question of determining if a subtiling must exist is similar to the question of existence of fault-free tilings. A fault in a tiling of a rectangular region $R_{n\times m}$ (typically with dominos) is a horizontal or vertical line that can be drawn through the interior of $R_{n\times m}$ and does not pass through the interior of any tile. A fault is called vertical or horizontal if the defining line is vertical or horizontal, respectively. A \emph{fault-free tiling} is a tiling with no fault. The problem of finding fault-free tilings is credited  to Robert I. Jewett in \cite{golomb1965polyominoes} where it was also shown that a fault-free rectangle must have length exceeding four, with a $5\times 6$ rectangle yielding the smallest construction. An elegant proof is given that no such tiling exists for the $6\times 6$ rectangle. 

Necessary and sufficient conditions for a rectangular region to admit a fault-free tiling of dominos has been worked out \cite{graham1981fault}. Later, the rectangular polyominoes that admit fault-free tilings was completely classified \cite{robinson1982fault}. The condition of having a subtiling is different as we ignore horizontal faults, consider sets of many different polyominoes, and we allow a tiling to be rearranged.

This paper is organized as follows: In Section \ref{sec:preliminaries}, we formalize our problem as well as proving that determining if a tiling admits a subtiling is $\mathsf{NP}$-complete when $\mcL$ only consists of rectangular polyominoes. We also prove that some conditions on $\mcL$ implies $\beta_n(\mcL)$ is finite using direct arguments. In  Section \ref{sec:toric}, we establish a connection between the finiteness of $\beta_n(\mcL)$ and certain rings being Noetherian. We then prove that if a finite set of polyominoes satisfies a technical condition, then we can show $\beta_n(\mcL)$ will be finite. In Section \ref{sec:examples}, we prove that certain families of finite sets of polyominoes actually satisfy the technical condition in Section \ref{sec:toric}. In Section \ref{sec:bounds}, we find upper bounds for $\beta_n(\mcL)$ for special sets $\mcL$.

\section{Preliminaries and first results}\label{sec:preliminaries}

Throughout this paper, we will have a finite set of polyominoes $\mcL$ which will be used to tile a rectangular region. While it is natural for such a set $\mcL$ to be closed under reflections and rotations, our proofs only require that the set be closed under reflections across a vertical line (we call this a \emph{vertical reflection}). Many of the examples we will consider, however, will only involve sets of rectangular polyominoes and  is thus closed under reflections.

\begin{defn}
 If $\beta^R_n(\mcL)<\infty$, we say that $\mcL$ is $n$-\emph{Noetherian}. If $\beta_n(\mcL)<\infty$, we say that $\mcL$ is \emph{strongly} $n$-Noetherian. Being strongly $n$-Noetherian implies being $n$-Noetherian as $\beta_n^R(\mcL)\le\beta_n(\mcL)$.
\end{defn}

The justification of terminology of $n$-Noetherian will be given later in Section \ref{sec:toric}. It is clear that for every $n$, there exists a set $\mcL$ that is (strongly) $n$-Noetherian: let $\mcL$ consist solely of the $1\times 1$ domino. We now give an example of a set $\mcL$ that is not $2$-Noetherian.

\begin{example}\label{ex:non-noetherian}
 Let $\mcL$ be the set consisting of the following three polyominoes and their vertical reflections.:
 
 \begin{center}
 \begin{tikzpicture}
  \draw (0,0) rectangle (.25,.25);
  \draw (.25,0) rectangle (.5,.25);
  \draw (0,-.25) rectangle (.25,0);
  \draw (.75,0) node {,};
  \draw (1.25,0) rectangle (1.5,.25);
  \draw (1,-.25) rectangle (1.25,0);
  \draw (1.25,-.25) rectangle (1.5,0);
  \draw (1.75,0) node{,};
  \draw (2,-.25) rectangle (2.25,0);
  \draw (2.25,-.25) rectangle (2.5,0);
  \draw (2.25,0) rectangle (2.5,.25);
  \draw (2.5,0) rectangle (2.75,.25);
 \end{tikzpicture}
 \end{center}
 \noindent We can construct tilings of $R_{2\times N}$, for $N$ arbitrarily large, such that no subtiling exists. Consider the following general tiling:
 
 \begin{center}
  \begin{tikzpicture}
   \draw (0,0) rectangle (.25,.25);
  \draw (.25,0) rectangle (.5,.25);
  \draw (0,-.3) rectangle (.25,0);
  
    \draw (.35,-.3) rectangle (.6,-.05);
  \draw (.6,-.3) rectangle (.85,-.05);
  \draw (.6,-.05) rectangle (.85,.25);
  \draw (.85,-.05) rectangle (1.1,.25);
  \draw (1.5,0) node {$\cdots$};
  
  \draw (1.75,-.3) rectangle (2,0);
  \draw (2,-.3) rectangle (2.25,0);
  \draw (2.,0) rectangle (2.25,.25);
  \draw (2.25,0) rectangle (2.5,.25);
  
  \draw (2.6,-.05) rectangle (2.85,.25);
  \draw (2.6,-.3) rectangle (2.85,-.05);
  \draw (2.35,-.3) rectangle (2.6,-.05);
  \end{tikzpicture}
 \end{center}

 \noindent It is clear that this tiling has no subtiling, even if we allow all reflections and rotations, and can be made arbitrarily wide. 
\end{example}

\subsection{The computational complexity of deciding if a subtiling exists}

In Example \ref{ex:non-noetherian}, we could visually tell that those particular tilings had no subtiling. However, this is not an easy problem in general. This is in contrast to checking to see if a tiling is fault-free, which is very easy. We will show that even if we restrict ourselves to rectangular polyominoes, this problem is hard.

We are given as input a tiling of a rectangle $R_{n\times m}$ for some $n$ and $m$ by rectangular polyominoes. The polyominoes in the tiling are given as ordered pairs of binary numbers $(h,w)$ where $h$ is the height and $w$ the width of the corresponding polyominoes. The unit square in $R_{n\times m}$ that the bottom left corner of each polyomino occupies in the tiling is also specified as an ordered pair of binary numbers with the convention that $(0,0)$ denotes the bottom left unit square in $R_{n\times m}$. 

\begin{problem}
 Define the problem $\mathsf{TRANS-ST}$ as the problem of deciding if a tiling of $R_{n\times m}$ by rectangular polynominoes admits a subtiling where rearrangement is restricted to translations.
 The problem $\mathsf{GEN-ST}$ asks if there exists a subtiling allowing rotations.
\end{problem}

We show that $\mathsf{TRANS-ST}$ and $\mathsf{GEN-ST}$ are $\mathsf{NP}$-complete. Many polyomino packing problems are $\mathsf{NP}$-complete and often can be reduced to the subset sum problem. Here we will reduce to the following related $\mathsf{NP}$-complete problem.

\begin{problem}
 The partition problem is the problem of deciding if  a multiset $M$ of positive integers can be written as the disjoint union $S_1\sqcup S_2$ where $\sum_{s\in S_1}{s}=\sum_{s\in S_2}{s}$.
\end{problem}

\begin{theorem}[\cite{karp1972reducibility}]
 The partition problem is $\mathsf{NP}$-complete.
\end{theorem}

We note that for rectangular polyominoes, reflections are trivial. As such, the difference in $\mathsf{TRANS-ST}$ and $\mathsf{GEN-ST}$ is only that in $\mathsf{GEN-ST}$ we can attempt to rotate pieces in our rearrangements.

\begin{theorem}
 The problems $\mathsf{TRANS-ST}$ and $\mathsf{GEN-ST}$ are both $\mathsf{NP}$-complete.
\end{theorem}
\begin{proof}
Clearly both problems are in $\mathsf{NP}$ as a subtiling can be easily verified. We wish to reduce both problems to the partition problem. Suppose we are given a multiset of positive integers $M=\{m_1,\dots,m_\ell\}$ and $2N=\sum_{m\in M}{m}$ (if $\sum_{m\in M}{m}$ is not even, we know no partition is possible). Then let $\mcL$ be the set of polyominoes consisting of the $(2N+1)\times m_i$ polyomino for every $1\le i\le \ell$ and also the $1\times N$ polyomino.  Then we tile the region $R_{(2N+2)\times 2N}$ by placing the two $1\times N$ polyominoes in the first row and one of each $(2N+1)\times m_i$ for $1\le i\le \ell$  polyomino side by side the the bottom $2N+1$ rows. This is clearly a tiling of $R_{(2N+2)\times 2N}$. 
 
 With this tiling, rearrangements allowing rotations are impossible (apart from rotations of $n\pi$ radians). First of all, we cannot rotate any of the $(2N+1)\times m_i$ polyominoes by $\pi/2$ radians as their height exceeds the width of the rectangular region being tiled. We can rotate the $1\times N$ polyominoes $\pi/2$ radians, however. Suppose we do this for one of the $1\times N$ polyominoes.
 
 It now lies within a single column whose remaining $N+2$ rows must be occupied by other polynominoes. However, none of the $(2N+1)\times m_i$ polyominoes can share a column with the $N\times 1$ polyomino as the height of the rectangular region being tiled is $2N+2$. So it can only share a column with the other $1\times N$ polyomino, or if we rotate it, an $N\times 1$ polynomio. But in either scenario, we see that some of the rows of the column in question will be left unoccupied.
 
 So this reduction is valid for both $\mathsf{TRANS-ST}$ and $\mathsf{GEN-ST}$ since we are using rectangular polyominoes and reflections are trivial.
 Lastly, we note that a subtiling exists for this tiling if and only if the polyominoes of the form $1\times m_i$ can be partitioned into two disjoint sets $S_1$ and $S_2$ such that the lengths of the polyominoes in $S_1$ sum to $N$ and likewise for the lengths of the polyominoes in $S_2$. As such, a subtiling exists if and only if $M$ has a partition of the desired kind.
\end{proof}

\subsection{The case of a single rectangular polyomino}

Returning our attention to the question of determining $\mcL$ such that $\beta^R_n(\mcL)<\infty$, we first look at the most basic case of a single rectangular polyomino. The result we prove allows all translations, reflections, and rotations of the $a\times b$ polyomino. Necessary and sufficient conditions for an  $a\times b$ rectangular polyomino to tile $R_{n\times m}$ have been worked out previously in the literature and this can be used to decide when a rectangular polyomino is $n$-Noetherian for a particular $n$. In particular, we have the following theorem.

\begin{theorem}[\cite{de1969filling,klarner1969packing}]\label{thm:rect_pack}
 An $a\times b$ polyomino tiles $R_{n\times m}$ if and only if the following conditions are satisfied.
 \begin{enumerate}[(a)]
  \item Either $m$ or $n$ is divisible by $a$, and either $m$ or $n$ is divisible by $b$.
  \item Both $m$ and $n$ can be written as linear combinations of $a$ and $b$ with coefficients in $\N$.
 \end{enumerate}
\end{theorem}

 In the following theorem, the property of being Noetherian actually comes from the fact that if we are given two natural numbers $a,b$ with greatest common divisor $d$, then for $m$ sufficiently large and divisible by $d$, it is always possible to write $m=ax+by$ for $x,y\in\N$ using the Euclidean algorithm. Let us denote the $a\times b$ polyomino by $p_{a\times b}$. The following result follows directly from Theorem \ref{thm:rect_pack}.

\begin{theorem}\label{thm:single_rectangle}
The set $\mcL=\{p_{a\times b}\}$ is $n$-Noetherian for every $n\in \N$.
\end{theorem}
\begin{proof}

There are three cases we consider. The first is that $n$ is divisible by neither $a$ nor $b$. Then if any tiling is to exist, $n=ax+by$, with $x,y\in \N$, otherwise no tilings exist and $p_{a\times b}$ is trivially $n$-Noetherian. If $n$ is of such a form, then if $p_{a\times b}$ tiles $R_{n\times m}$, $m$ must be divisible by $ab$ and can trivially be written as a linear combination of $a$ and $b$ with coefficients in $\N$. Then if $m>ab$ and $p_{a\times b}$ tiles $R_{n\times m}$, we have that $m=dab$ for $d>1$. Let $d=(d'+d'')$, $d',d''\ge 1$ and $m'=d'ab$ and $m''=d''ab$. Then by Theorem \ref{thm:rect_pack}, $R_{n\times m'}$ and $R_{n\times m''}$ are both tilable by $p_{a\times b}$ and so for any tiling of $R_{n\times m}$ with $m>ab$, the tiling admits a subtiling allowing rotations.

Now suppose $n$ is divisible by exactly one of $a$ and $b$, without loss of generality let it be $a$. Then trivially $n$ can be written as linear combination of $a$ and $b$ with coefficients in $\N$. Then if $p_{a\times b}$ tiles $R_{n\times m}$, we have that $b$ divides $m$. Then we repeat the argument above for when $m>b$. 

The last case is when $ab$ divides $n$. Once again, $n$ can be written as a linear combination of $a$ and $b$ with coefficients in $\N$. Let $\gcd(a,b)=d$ . Then if $p_{a\times b}$ tiles $R_{n\times m}$, we have that $m$ is divisible by $d$ by condition (b) of Theorem \ref{thm:rect_pack}.  There exists an $N\in\N$ (which we may assume is divisible by $d$) such that for every $m\ge N$ divisible by $d$, $m=ax+by$ for some $x,y\in\N$. Thus for $m\ge 2N$ divisible by $d$, we know that $R_{n\times m}$ admits a subtiling allowing rotations since we can choose $m',m''\ge N$, both divisible by $d$, and $R_{n\times m'}$ and $R_{n\times m''}$ are tilable by $p_{a\times b}$ by Theorem \ref{thm:rect_pack}.
\end{proof}

It is not too difficult to come up with bounds for $\beta_n(p_{a\times b})$. If $n$ is not divisible by $a$ nor $b$, then $\beta_n(p_{a\times b})=2ab$. Similarly, if $n$ is divisible by precisely one of $a$ or $b$, say $a$, then $\beta_n(p_{a\times b})=2b$. However, if $ab$ divides n, it is a bit trickier to determine $\beta_n(p_{a\times b})$.

\subsection{Sets of Tall Rectangular Polynominoes}

There is another family of rectangular polyominoes we can show are strongly $n$-Noetherian by elementary arguments. Let $\mcL$ be a set of rectangular polyominoes such that every element has height $>n/2$ and at most one polyomino of unit height. Let the polyomino of unit height have width dividing the greatest common divisor of the widths of all of the other polyominoes in $\mcL$. Then we have the following theorem.

\begin{theorem}
 If $\mcL$ is as above, then it is strongly $n$-Noetherian. Furthermore, $\beta_n(\mcL)$ equals the largest width of any polyomino in $\mcL$.
\end{theorem}
\begin{proof}
 Let us first focus on those polyominoes of height $>n/2$. It is clear that no two such polyominoes can occupy space in the same column of $R_{n\times m}$. As such, we shall arrange these polyominoes in $R_{n\times m}$ as follows.
 
 
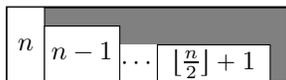
\begin{center}
 \begin{tikzpicture}
 \draw[thick,fill=gray] (0,0) rectangle (3.75,1);
  \draw[fill=white] (0,0) rectangle (.5,1);
  \draw (.25,.5) node{$n$};
  \draw[fill=white,draw=none] (1.5,0) rectangle (2,.5);
  \draw[fill=white] (.5,0) rectangle (1.5,.75);
  \draw (1,.4) node{$n-1$};
  
  \draw (1.75,.25) node{$\cdots$};
  \draw[fill=white] (2,0) rectangle (3.5,.5);
  \draw (2.75,.25) node{$\lfloor \frac{n}{2}\rfloor+1$};
   \end{tikzpicture}
\end{center}

In the picture, the rectangle labeled $n$ denotes a contiguous block of height $n$ rectangular polynominoes. The same for $n-1$, all the way down to $\lfloor\frac{n}{2}\rfloor+1$. The grey area is then filled with the copies of the polyomino of unit height. This can always be done as the width of this polyomino divides the width of any of the polyominoes of height $>n/2$ and thus divides the width of any of these contiguous blocks. 

This rearrangement admits a subtiling. The subtiling $R_{n \times m'}$ can be formed by taking any polyomino $p$ of height $>n/2$ and tiling the space above it with the polyomino of unit height; the remaining rectangle $R_{n\times m''}$ may be tiled in any way. The tiling of $R_{n\times m'}$ described admits no subtiling and $m'$ must in fact be the width of $p$. This shows that $\beta_n(\mcL)$ is in fact the maximum width of any rectangular polyomino in $\mcL$ of height $>n/2$.
\end{proof}

\section{Rephrasing the problem in terms of rings generated by monomials}\label{sec:toric}

We now recast this problem in terms of rings, justifying the use of the term ``Noetherian". Let us first restrict ourselve to rearrangements only allowing translations. Let $\mcL$ be a finite collection of polyominoes. We view each polyomino as an algebraically independent variable. We can form the formal commutative ring over $\C$ generated by elements of $\mcL$, denote it $\C[\mcL]$.

However, let us consider the subring $\mcR_n[\mcL]\subseteq \C[\mcL]$, whose homogeneous elements are  equivalence classes of tilings of $R_{n\times m}$ for a fixed $n$. Two tilings are equivalent if one can be rearranged to give the other using translations. This is equivalent to the commutativity of $\C[\mcL]$. Equivalence classes of tilings are monomials in the ring $\C[\mcL]$. In the following discussion we will say ``tiling" instead of ``equivalence class of tilings" when it is clear from context.

Given tilings of $R_{n\times m_1}$ and $R_{n\times m_2}$, denoted $T_1$ and $T_2$, respectively, we define $T_1*T_2$ to be the tiling of $R_{n\times (m_1+m_2)}$ given by adjacent juxtaposition of $T_1$ and $T_2$. There are two ways to naturally do this, but they are in the same equivalence class and so this product is well defined.

Now suppose we wanted to allow reflections and rotations in our rearrangements. Then we would make sure that $\mcL$ is closed under rotations and reflections. Then we consider a quotient ring of $\mcR_n[\mcL]$ by imposing the relation that any polyomino $p$ (viewed as a variable in $\C[\mcL]$) is equal to the polyominoes that can be reached by rotating and reflecting $p$. We denote this ring by $\tilde{\mcR}_n[\mcL]$.

It is clear that if $\mcL$ is strongly $n$-Noetherian then $\mcR_n[\mcL]$ is Noetherian. This is because every tiling of $R_{n\times m}$ for $m>\beta_n(\mcL)$ can be written as a product of tilings of $R_{n\times m'}$ and $R_{n\times m''}$ with $m=m'+m''$. So $\mcR_n$ is generated by finitely many monomials. Likewise, $\tilde{R}_n[\mcL]$ is Noetherian if $\mcL$ is $n$-Noetherian. We now show that the converse holds as well. 

\begin{proposition}\label{prop:just}
 The ring $\mcR_n[\mcL]$ is Noetherian if and only if $\mcL$ is strongly $n$-Noetherian. Likewise, $\tilde{\mcR}_n[\mcL]$ is Noetherian if and only if $\mcL$ is $n$-Noetherian.
\end{proposition}
\begin{proof}
We have already discussed one direction for the two rings. Now let us suppose that $\mcR_n[\mcL]$ is Noetherian. Since $\mcR_n[\mcL]$ is ring generated by monomials in algebraically independent variables, it is well know that the ideal of  relations are generated by those of the form $X_1*X_2=Y_1*Y_2$ where $X_1$, $X_2$, $Y_1$, and $Y_2$ are monomials. These are called binomial relations. This corresponds to the observation that a monomial may be factored into smaller monomials non-uniquely. Since the ideal of relations is a homogeneous ideal, if $\mcR_n[\mcL]$ is Noetherian, then it must be the case that for sufficiently large $N$, if $T$ is monomial of degree at least $N$ then $T=T_1*T_2$. This implies that $\mcL$ is strongly $n$-Noetherian. For the ring $\tilde{\mcR}_n[\mcL]$, the argument is the same where we note that the kernel of the map taking $\mcR_n[\mcL]$ to $\tilde{\mcR}_n[\mcL]$ was also generated by binomial relations.
\end{proof}

Proving the ring $\C[\mcL]$ is Noetherian for various $\mcL$ is too hard, so we instead use another ring whose Noetherianess will imply $\C[\mcL]$ is Noetherian. This ring will not be isomorphic to $\C[\mcL]$ in general and will restrict the translations we can make in rearranging tilings to only horizontal translations. To construct this new ring, we assume that $\mcL$ only contains rectangular polyominoes.

Given $n\in\N$, let $I(n):=\{(\ell,\ell+1,\dots,\ell+k)|\;\ell,\ell+k\in [n]\}$, where $[n]=(1,\dots,n)$, be the set of all contiguous intervals contained in $[n]$. For an interval $\iota=(\ell,\dots, \ell+k)\in I(n)$, let $|\iota|=k$ be the length of that interval. Given a polyomino of height $k$, we will create several different variables associated to it, one for each element of $I(n)$ of length $k$. Let $\omega$ be the maximum width of any polyomino in $\mcL$. Let us define the following vector space:
$$V(n):=\bop_{\iota\in I(n)}{\End(V_{\iota})},$$
$$\End(V_{\iota}):=\ott_{j=1}^{|\iota|}{\End(\C^{\omega})}.$$

The coordinate ring of $V(n)$ is isomorphic to $\ott_{\iota\in I(n)}{\C[\End(V_{\iota})]}$, where the ring $\C[\End(V_{\iota})]$ is the coordinate ring of $\End(V_{\iota})$.
Let us fix $\iota\in I(n)$. Then for every polyomino $p\in \mcL$ of height $|\iota|$, we associate to it the polynomial $\Tr(M_{\iota}^w)$ where $w$ is the width of $p$ and $M_{\iota}$ is the matrix of variables in $\End(V_{\iota})$. Then let $\C[\mcL]_{\iota}\subseteq \C[\End(V_{\iota})]$ be the subring generated by $1$ and these polynomials. Then we define $\mcP_n[\mcL]:=\ott_{\iota\in I(n)}{\C[\mcL]_\iota}$ which is a subring of the coordinate ring of $V(n)$.

We can think of $\mcP_n[\mcL]$ as a ring generated by tuples $(p,\iota)$ where $p\in\mcL$ and $\iota$ is a tuple containing the rows of $R_{n\times m}$ $p$ must be placed in. For this reason, the rearrangements will only consist of horizontal translations since row data is implicit in each variable. We would now like to show that these variables are algebraically independent. For that we need the following result.

\begin{theorem}[\cite{procesi1976invariant,kraft2000classical}]\label{thm:traceind}
 For the generic matrix $M\in \End(\C^n)$, the polynomials $\Tr(M),\dots,\Tr(M^n)$ are algebraically independent.
\end{theorem}

\begin{proposition}\label{prop:poly}
 The ring $\mcP_n[\mcL]$ is a Noetherian polynomial ring.
\end{proposition}
\begin{proof}
 That $\mcP_n[\mcL]$ is Noetherian is clear as it has finitely many generators. Since $\mcP_n[\mcL]=\ott_{\iota\in I(n)}{\C[\mcL]_\iota}$, it suffices to show that $\C[\mcL]_\iota$ is a polynomial ring for any $\iota\in I(n)$. We know that $\C[\mcL]_\iota$ is generated by polynomials of the form $\Tr(M_\iota^w)$ where $w\le\omega$. The matrix $M_\iota$ is the generic $(|\iota|\cdot\omega)\times (|\iota|\cdot\omega)$ matrix and so by Theorem \ref{thm:traceind}, these polynomials are algebraically independent. Thus $\C[\mcL]_{\iota}$ is a polynomial ring.
\end{proof}

Consider the vector space $\End(\C^{\omega n})\cong \ott_{i=1}^n{\End(\C^\omega)}:=\ott_{i=1}^n{\End(W_i)}$. Let $(\lambda_1,\dots,\lambda_{n-1})\in(\C^{\times})^{n-1}$ act on $\End(W_i)$ by multiplication by $\lambda_i$ if $i\ne n$ and otherwise by multiplication by $(\lambda_1\cdots\lambda_{ n-1})^{-1}$. This induces an action on $\mcP_n(\mcL)$ as follows. Given the generator $\Tr(M_{\iota}^w)$, $(\lambda_1,\dots,\lambda_{n-1})$ multiplies this polynomial by $$\prod_{i\in\iota}{\lambda_i^w}\tn{ if }n\notin \iota\tn{ and}$$
$$(\lambda_1\cdots\lambda_{n-1})^{-w}\prod_{i\in \iota\setminus \{n\}}{\lambda_i^w}\tn{ otherwise.}$$ The group $(\C^{\times})^{n-1}$ is a semi-simple algebraic group and this action is clearly rational. This means we can use Hilbert's famous result on invariant rings.

\begin{theorem}[\cite{hilb:93,hilbert1890theorie}]\label{thm:hilbnoether}
 For a commutative Noetherian $\C$-algebra $R$ acted on rationally by semi-simple algebraic group $G$, the invariant ring $R^G$ is Noetherian.
\end{theorem}

We now look at the closely related action of $(\lambda_1,\dots,\lambda_n)\in(\C^{\times})^{n}$ on $\End(W_i)$ by multiplication by $\lambda_i$, even when $i=n$. This action again extends to an action on $\mcP_n[\mcL]$. This action induces a multigrading on the ring. Given a monomial $X\in\mcP_n[\mcL]$, we say that is is \emph{multihomogeneous} of degree $(w_1,\dots,w+n)$ if $(\lambda_1,\dots,\lambda_n).X=(\lambda_1^{w_1}\cdots\lambda_n^{w_n}) X$.

\begin{defn}
 We say that a set of rectangular polyominoes $\mcL$ is $n$-\emph{representable} if $p_1\cdots p_\ell$ is a monomial in $\mcP_n[\mcL]$ of homogeneous degree $(w_1,\dots,w_m)$, then there is a tiling of $R_{n\times m}$, $m=\sum_{i=1}^n{w_i}$ by the polyominoes associated to $p_1,\dots,p_\ell$ which respects their row assignments.
\end{defn}

In the above definition, we remind the reader that polynomials $p_1,\dots,p_\ell$ correspond to polyominoes and a designation of which rows the polyomino must be placed in. We need that multihomogeneous polynomials always correspond to tilings of rectangular regions. We can now present our main result.

\begin{theorem}\label{thm:main}
 If $\mcL$ is $n$-representable, then $\mcL$ is strongly $n$-Noetherian.
\end{theorem}
\begin{proof}
 The ring $\mcP_n[\mcL]$ is Noetherian and by Theorem \ref{thm:hilbnoether}, so is $\mcP_n[\mcL]^{G_n}$, where $G_n=(\C^{\times})^{n-1}$ acts on $\mcP_n[\mcL]$ as described above. It is clear that the homogeneous elements of $\mcP_n[\mcL]^G$ are the multihomogeneous elements of degree  $(m,\dots,m)$ for $m\in \N$. Then if $\mcL$ is $n$-representable, the homogeneous elements of $\mcP_n[\mcL]^{G_n}$ correspond to equivalence classes of tilings of rectangular regions. Since this ring is Noetherian and generated by monomials of algebraically independent variables (by Proposition \ref{prop:poly}), we can use the same argument as in the proof of Proposition \ref{prop:just} to say that every monomial of sufficiently large degree must factor into two monomials of smaller degree. This implies that $\mcL$ is $n$-Noetherian and furthermore, the arrangements only allow horizontal translations. So it is strongly $n$-Noetherian.
\end{proof}

\section{Examples of representable sets of polyominoes}\label{sec:examples}

In the previous section we showed that $n$-representable sets of polyominoes are strongly $n$-Noetherian. However, the condition of being $n$-representable is rather opaque. In this section, we discuss some concrete sets of polyominoes that are $n$-representable for specific $n$.

Let us first expound a bit further on the definition of $n$-representable. Suppose we are given a multiset $\mcP=\{(p_{a\times b},\iota)|\;p_{a\times b}\in\mcL,\;\iota\in I(n)\}$ where $\mcL$ is finite set of rectangular polyominoes. As before, we use $p_{a\times b}$ to mean the $a\times b$ rectangular polyomino. The tuple $(p_{a\times b},\iota)$ indicates the rows in $R_{n\times m}$ that $p_{a\times b}$ should span. Then suppose the following conditions hold:
\begin{equation}
\begin{split}
&\sum_{(p_{a\times b},\iota)\in\mcP}{ab}=nm\tn{ and}\\
&\sum_{\stackrel{(p_{a\times b},\iota)\in\mcP}{i\in\iota}}{b}=m\tn{ for all }1\le i\le n.
\end{split}
\end{equation}\label{eq:rep}
Then the set $\mcL$ is $n$-representable if whenever the above two conditions are met, $\mcP$ tiles $R_{n\times m}$, respecting the row assignments of the polyominoes, 

\begin{observation}
 Note that if $\mcL$ consists only of polyominoes of the form $1\times k$, it is $n$-representable for any $n$.
\end{observation}

\begin{example}\label{ex:not-rep}
We now give an example of a set of polyominoes that are not 5-representable. Let $\mcL=\{p_{1\times 3},p_{1\times 2},p_{1\times 1},p_{3\times 1}\}$. Now consider the following set
$$\mcP=\{(p_{1\times 3},\{1\}),(p_{1\times 3},\{5\}),(p_{3\times 1},\{1,2,3\}),(p_{3\times 1},\{2,3,4\}),$$
$$(p_{3\times 1},\{3,4,5\}),(p_{1\times 1},\{3\}),(p_{1\times 2},\{2\}),(p_{1\times 2},\{4\})\}.$$
We see that sum of the areas of the polyominoes in $\mcP$ is $3*5+2*2+1=4*5$. We also see that
$$\sum_{\stackrel{(p_{a\times b},\iota)\in\mcP}{i\in\iota}}{b}=4\tn{ for all }1\le i\le 5.$$
If $\mcL$ were 5-representable, we should be able to tile $R_{5\times 4}$ with polyominoes in $\mcP$ satisfying the row assignments. We show that this is not the case. First note that there is a $1\times 3$ polyomino in both the top and bottom row. Also, there is a $3\times 1$ polyomino that intersects the first row, and another $3\times 1$ polyomino that intersects the last row. So if we could tile $R_{5\times 4}$, there are only two ways that these four tiles could be placed. One way is shown below. The other way is the vertical reflection of the following picture.

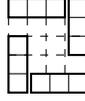
\begin{center}
 \begin{tikzpicture}
  \draw[thick] (0,0) rectangle (.75,.25);
  \draw (.25,0) -- (.25,.25);
  \draw (.5,0) -- (.5,.25);
  \draw[thick] (.8,.25) rectangle (1.05,-.5);
  \draw (.8,0) -- (1.05,0);
  \draw (.8,-.25) -- (1.05,-.25);
  \draw[thick] (0,-.25) rectangle (.25,-1);
  \draw (0,-.5) -- (.25,-.5);
  \draw (0,-.75) -- (.25,-.75);
  \draw[thick] (.3, -1) rectangle (1.05,-.75);
  \draw (.55,-1) -- (.55,-.75);
  \draw (.8,-1) -- (.8,-.75);
  \draw[dashed] (0,-.25) -- (1.05,-.25);
  \draw[dashed] (0,-.5) -- (1.05,-.5);
  \draw[dashed] (0,0) -- (0,-.25);
  \draw[dashed] (.25,0) -- (.25,-.25);
  \draw[dashed] (.5,0) -- (.5,-.75);
  \draw[dashed] (.75,0) -- (.75,-.75);
  \draw[dashed] (1.05,-.5) -- (1.05,-.75);
 \end{tikzpicture}
\end{center}
We now see that there are two ways to place the last $3\times 1$ polyomino, which must span rows 2,3, and 4. But either way we do this, we make it impossible to place one of the two $1\times2$ polyominoes. So we cannot tile $R_{5\times 4}$ using these polyominoes while respecting their row assignments. So $\mcL$ is not 5-representable.
\end{example}

In Example \ref{ex:not-rep}, we showed that set was not 5-representable by using polyominoes to construct a region inside of $R_{5\times 4}$ that could not be tiled with polyominoes of unit height. If we want to prove that sets of polyominoes are $n$-representable, we need to understand which regions of $\R^2$ can always be tiled with polyominoes of unit height.

\begin{defn}
 A subregion $R\subseteq R_{n\times m}$ is called \emph{row-convex} if for any two points $x,y\in R$ that can be connected by a horizontal line, the convex hull of those two points lies entirely within $R$. $R$ inherits rows and columns from $R_{n\times m}$. Let $|R_i|$ denote the width of $R_i$, the $i$th row of $R$.
\end{defn}

\begin{lemma}\label{lem:row-convex}
Let $R$ be a row-convex region and consider any multiset of the form $\mcP=\{(p_{a\times b},\iota)|\;p_{1\times b}\in\mcL,\;\iota\in I(n)\}$. Then if
$$\sum_{\stackrel{(p_{1\times b},\iota)\in\mcP}{i\in\iota}}{b}=|R_i|\tn{ for all }1\le i\le n,$$ there is a tiling of $R$ by $\mcP$ respecting row assignments.
\end{lemma}
\begin{proof}
Since $$\sum_{\stackrel{(p_{1\times b},\iota)\in\mcP}{i\in\iota}}{b}=|R_i|\tn{ for all }1\le i\le n,$$ we can simply place those polyominoes assigned to row $i$ in any order we want since $R$ is row-convex. The tiling of each row does not affect any other row, so $R$ is tilable by $\mcP$ respecting row assignments.
\end{proof}

If $\mcL$ is a set of rectangular polyominoes, $\mcL$ is $n$-representable if when tiling $R_{n\times m}$ and Equations \ref{eq:rep} hold, those polyominoes which are not unit height can be always be arranged such that the remaining untiled region is row-convex. Then appealing to Lemma \ref{lem:row-convex}, $R_{n\times m}$ can be tiled.

\begin{theorem}\label{thm:n-rep}
For $n\in\N$, if $\mcL$ consists of rectangular polyominoes whose height is either 1 or at least $n-1$, then $\mcL$ is $n$-representable. Furthermore, if $n=4$, then any set of rectangular polyominoes where all of the polynominoes of height two have the same width is $n$-representable.
\end{theorem}
\begin{proof}
 Let $\mcP$ be a multiset satisfying Equations \ref{eq:rep} for $n,m\in\N$. To prove the theorem, we simply need to show that those polyominoes of height $\ge n-1$ can be translated horizontally in their assigned rows in $R_{n\times m}$ such that the remaining untiled region is row-convex. Then we appeal to Lemma \ref{lem:row-convex}.
 
 Let us first handle the case that all polyominoes have at at least $n-1$ or $1$. We then arrange those polyominoes in $\mcP$ as in the following picture.
 
 
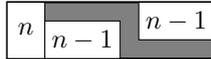
\begin{center}
  \begin{tikzpicture}
  \draw[thick,fill=gray] (0,0) rectangle (2.75,.75);
   \draw[fill=white] (0,0) rectangle (.5,.75);
   \draw (.25,.4) node {$n$};
   \draw[fill=white] (.5,0) rectangle (1.5,.5);
   \draw (1,.25) node{$n-1$};
   \draw[fill=white] (1.75,.25) rectangle (2.75,.75);
   \draw (2.25,.5) node{$n-1$};
  \end{tikzpicture}
 \end{center}
Here the numbers in the boxes indicate that the box is a contiguous block of polyominoes of that height. Every possible row assignment of every polyomino of height $n-1,n$ is shown in this picture. The grey region is the remaining untiled area. As we can see, it is row convex. Therefore, any set of rectangular polyominoes meeting the specified height conditions is $n$-representable.

Now let $n=4$. We once again need to rearrange the polyominoes of height $\ge 2$ in such a way that a row-convex region is left untiled. There are two scenarios, both pictured below.

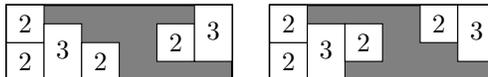
\begin{center}
 \begin{tikzpicture}
 \draw[thick,fill=gray] (0,0) rectangle (3,1);
  \draw[fill=white] (0,0) rectangle (.5,.5);
  \draw[fill=white] (0,.5) rectangle (.5,1);
  \draw (.25,.25) node{2};
  \draw (.25,.75) node{2};
  \draw[fill=white] (.5,0) rectangle (1,.75);
  \draw (.75,.4) node{3};
  \draw[fill=white] (1,0) rectangle (1.5,.5);
  \draw (1.25,.25) node{2};
  \draw[fill=white] (2,.25) rectangle (2.5,.75);
  \draw (2.25,.5) node{2};
  \draw[fill=white] (2.5,.25) rectangle (3,1);
  \draw (2.75,.65) node {3};
  
 \draw[thick,fill=gray] (3.5,0) rectangle (6.5,1);
  \draw[fill=white] (3.5,0) rectangle (4,.5);
  \draw[fill=white] (3.5,.5) rectangle (4,1);
  \draw (3.75,.25) node{2};
  \draw (3.75,.75) node{2};
  \draw[fill=white] (4,0) rectangle (4.5,.75);
  \draw (4.25,.4) node{3};
  \draw[fill=white] (4.5,.25) rectangle (5,.75);
  \draw (4.75,.5) node{2};
  \draw[fill=white] (5.5,.5) rectangle (6,1);
  \draw (5.75,.75) node{2};
  \draw[fill=white] (6,.25) rectangle (6.5,1);
  \draw (6.25,.65) node{3}; 
 \end{tikzpicture}
\end{center}

Once again, the boxes represent contiguous blocks of polyominoes of the designated height and every row assignment is accounted for. The grey region is the part left untiled and in both pictures is row-convex. We note that we need the assumption that all polyominoes of height two have the same width to ensure that the two height two regions on the left hand side of the picture can be assumed to have equal width. 

\end{proof}

\begin{corollary}
 For $n\in [3]$, every set of rectangular polyominoes is strongly $n$-Noetherian. Furthermore, if $\mcL$ consists only of rectangular polyominoes of height 1 or height at least $n-1$, then $\mcL$ is strongly $n$-Noetherian. Lastly if $n=4$ and all polyominoes of height two in $\mcL$ have the same width, it is $4$-Noetherian.
\end{corollary}
\begin{proof}
 This follows immediately from Theorem \ref{thm:main} and Theorem \ref{thm:n-rep}.
\end{proof}

It is interesting to note that as long as $n\le 3$, every set of rectangular polyominoes is $n$-representable. However, this does not continue to hold beyond $n=5$ as was shown in Example \ref{ex:not-rep}. Examples can also be found of sets of  rectangular polyominoes that are not 4-representable. However, we do not know if Example \ref{ex:not-rep} is strongly 5-Noetherian or not. 

\section{Finding upper bounds on subtilings}\label{sec:bounds}

For an $n$-representable set $\mcL$, finding a bound on $\beta_n(\mcL)$ is equivalent to bounding the degree of the generators of the corresponding ring $\mcP_n[\mcL]^{G_n}$, where as before $G_n=(\C^{\times})^{n-1}$. There are results from invariant theory about bounding the degrees of generators. We look at the most general theorem first, before specializing to a particular case. For the sake of brevity, we assume some basic knowledge of algebraic geometry in the subsequent exposition.

We first need to realize $G_n$ as an affine variety. The set $\C^{n-1}$ is clearly an affine variety with coordinate ring $\C[x_1,\dots,x_{n-1}]$. Then the ring $\C[x_1,\dots,x_{n-1},t]/\mcI$, where $\mcI=\langle x_it-1, i\in[n-1]\rangle$, is the coordinate ring of $(\C^{\times})^{n-1}$, as each element must now be invertible. So $G_n$ is an affine subvariety of $\C^n$ defined by quadratic polynomials.

Given a rational representation $\rho:G\to\GL(V)$, we have a map $g\mapsto\{a_{ij}(g)\}$ where $\{a_{ij}(g)\}$ is a matrix in $\GL(V)$ and $a_{ij}(g)$ are rational functions in $g$. Let $A_\rho=\tn{max}(\tn{deg}(a_{ij}))$. 

\begin{theorem}[\cite{Derksen2001polynomial}]\label{thm:badbound}
 Let $G$ be an affine variety of $\C^n$ defined by polynomials $h_1,\dots,h_\ell$. Let $\rho:G\to\GL(V)$ be a rational representation of $G$ with finite kernel. Define $H=\tn{max}(\tn{deg}(h_i))$ and let $m=\dim(G)$. Then $\C[V]^G$ is generated by polynomials of degree at most $$\tn{max}(2,\dfrac{3}{8}\dim(\C[V]^G)H^{n-m}A_\rho^m).$$
\end{theorem}

In the above theorem, the Krull dimension of $\C[V]^G$ appears in the formula. In general, we do not know this dimension for $\mcP_n[\mcL]^{G_n}$, but is upper bounded by the Krull dimension of $\mcP_n[\mcL]$. Since $\mcP_n[\mcL]$ is polynomial ring by Proposition \ref{prop:poly}, $\tn{Spec}(\mcP_n[\mcL])$ is isomorphic to affine space and the Krull dimension of $\mcP_n[\mcL]$ is equal to the dimension of this space.

Given a set of rectangular polyominoes $\mcL$, let us define $h_n(\mcL)=\{a\le n|\;p_{a\times b}\in\mcL\}$ and $\mcL(a)=\{p_{a\times b}|\;p_{a\times b}\in \mcL\}$. Note that the polyomino $p_{a\times b}$ can be placed in $n-a+1$ different contiguous rows in $R_{n\times m}$, provided $a\le n$. Since $\mcP_n[\mcL]=\ott_{\iota\in I(n)}{\C[\mcL]_\iota}$, $\tn{Spec}(\mcP_n[\mcL])$ is isomorphic to $\bop_{\iota\in I(n)}{\tn{Spec}(\C[\mcL]_\iota)}$. Since $\tn{Spec}(\C[\mcL]_\iota)$ is also isomorphic to affine space, by Proposition \ref{prop:poly}, its dimension is the number of polyominoes in $\mcL$ of height $|\iota|$. So $\dim(\mcP_n[\mcL])=\sum_{i\in h_n(\mcL)}{(n-i+1)|\mcL(i)|}$.

\begin{theorem}\label{thm:toricbound}
 Given an $n$-representable set $\mcL$, let $A$ be the maximum area of a polyomino in $\mcL$ with height not exceeding $n$, then $$\beta_n(\mcL)\le\tn{max}\bigg(2,\dfrac{3}{4}A^{n-1}\cdot\sum_{i\in h_n(\mcL)}{(n-i+1)|\mcL(i)|}\bigg)$$
\end{theorem}
\begin{proof}
 This will follow from Theorem \ref{thm:main} and Theorem \ref{thm:badbound} after we substitute in the appropriate quantities. We note that $H=2$ as $G_n$ is defined by quadratics. Furthermore, $m=n-1$ and $t=n$. By the way $G_n$ acts on $\mcP_n[\mcL]$, we see that the maximum degree of a monomial is equal to the maximum area of polyomino in $\mcL$ with height not exceeding $n$. Thus $A_\rho=A$. Lastly $$\mcP_n[\mcL]^{G_n}\le\sum_{i\in h_n(\mcL)}{(n-i+1)|\mcL(i)|}.$$
\end{proof}

\subsection{The case where all polyominoes have unit height}

Let us look at the special case where $\mcL$ consists only of polyominoes of the form $1\times k$. As already observed, such a set is $n$-representable for every $n$. For this special case, we can actually realize $\mcP_n[\mcL]^{G_n}$ in a different way. Let $n$ be fixed and $\omega$ the maximum width of a polyomino in $\mcL$. Now we look at the subvariety $\mathcal{V}\subset\End((\C^{\omega})^{\ot n})$ of matrices of the form $\ott_{i=1}^n{M_i}$ where each $M_i\in\End(\C^\omega)$. 

Now let $\GL_{n,\omega}:=\times_{i=1}^n{\GL(\C^\omega)}$ act on $\mathcal{V}$ by $$(g_1,\dots,g_n).\ott_{i=1}^n{M_i}:=\ott_{i=1}^n{g_iM_ig_i^{-1}}.$$
If $\C[\mathcal{V}]$ denotes the coordinate ring of $\mathcal{V}$, then $\C[\mathcal{V}]^{\GL_{n,\omega}}$ is the corresponding invariant ring. We claim that $\mcP_n[\mcL]^{G_n}$ is isomorphic to a quotient ring of $\C[\mathcal{V}]^{\GL_{n,\omega}}$.

\begin{defn}
 For $\sigma=(\sigma_1,\dots,\sigma_n)\in\mcS_m^n$, let $\sigma_i=(r_1\cdots r_k)(s_1\cdots s_l)\cdots$ be a disjoint cycle decomposition.  
For such a $\sigma\in\mcS_m^n$, define $\tn{Tr}_\sigma=T_{\sigma_1}\cdots T_{\sigma_n}$ on $\mathcal{V}$, where 
$$T_{\sigma_i}(\ott_{j=1}^n{M_{j}})=\tn{Tr}(M_i^k)\tn{Tr}(M_i^l)\cdots.$$
\end{defn}

We see that the functions in the above definition are precisely the polynomial invariants of $\mcP_n[\mcL]^{G_n}$ when $\mcL$ is the set of rectangular polyominoes $\{p_{1\times k}|\; k\in[\omega]\}$.

\begin{theorem}[\cite{turner2015complete}]\label{thm:glinvs}
 The invariant ring $\C[\mathcal{V}]^{\GL_{n,\omega}}$ is generated by the polynomials $\Tr_{\sigma}$ for $\sigma\in\mcS_m^n$ such that the disjoint cycle decomposition of $\sigma$ contains only cycles of length at most $\omega$.
\end{theorem}

The above theorem tells us that if $\mcL=\{p_{1\times k}|\; k\in[\omega]\}$, then $\mcP_n[\mcL]^{G_n}=\C[\mathcal{V}]^{\GL_{n,\omega}}$. In general, let us define $w(\mcL)=\{k|\; 1\times k\in\mcL\}$ and take the subvariety of $\mathcal{V}$ such that for any $M\in\mathcal{V}$ and any decomposition M=$\ott_{i=1}^n{M_i}$, $\Tr(M_i^\ell)=0$ for $\ell\in[\omega]$, $\ell\notin w(\mcL)$, for all $1\le i\le n$. This defines another $\GL_{n,\omega}$-stable subvariety of $\mathcal{V}$, denote it $\mathcal{V}_{\mcL}$. This means that $\C[\mathcal{V}_{\mcL}]^{\GL_{n,\omega}}$ is a quotient ring of $\C[\mathcal{V}]^{\GL_{n,\omega}}$ \cite{nagata1963invariants}. As such, a degree bound on $\C[\mathcal{V}]^{\GL_{n,\omega}}$ gives a degree bound on $\C[\mathcal{V}_{\mcL}]^{\GL_{n,\omega}}\cong\mcP_n[\mcL]^{G_n}$.

We can use some classical tools from invariant theory to find  bounds for this ring. We introduce some definitions.

\begin{defn}
 The \emph{null cone} of an action $G\actson V$ is the set vectors $v$ such that $0\in\overline{G.v}$ (the Zariski closure of the orbit of $v$). We denote it by $\mcN_V$. Equivalently, $\mcN_V$ are those $v\in V$ such that $f(v)=f(0)$ for all invariant polynomials $f$. 
\end{defn}

If one wants to compute the generators of an invariant ring, this can be accomplished by determining the null cone. If one knows precisely the generators of the ideal of the null cone, then these are precisely the generators of the invariant ring \cite{hilb:93,derksen1997constructive}. However, it usually easier to find sets of invariant polynomials whose vanishing locus is $\mcN_V$. For that, one can use the following result (similar to Theorem \ref{thm:badbound}).

\begin{theorem}[\cite{derksen2013computational}]\label{thm:nullfinite}
Let $G\actson V$, be a reductive group acting rationally on a vector space. Let $f_1,\dots,f_\ell$ be homogeneous invariants, with maximum degree $p$, such that their vanishing locus is $\mcN_V$. Then $\C[V]^G$ is generated by polynomials of degree at most $$\operatorname{max}(2,\frac{3}{8}\dim(\C[V]^G)p^2).$$ 
\end{theorem}

In the proof of the above Theorem, the same arguments can be used to prove the following slightly stronger statement that we shall need.
\begin{corollary}\label{cor:nullfinite}
 Let $G\actson V$ be a reductive group acting rationally on a vector space and $X\subseteq V$ be a $G$-stable subvariety. If $\C[X]^G$ is Cohen-Macaulay, then let $f_1,\dots,f_\ell$ be homogeneous invariants of maximum degree $p$ whose vanishining locus is the null cone of the action of $G$ on $X$. Then $\C[X]^G$ is generated by polynomials of degree at most $$\operatorname{max}(2,\frac{3}{8}\dim(\C[V]^G)p^2).$$ 
\end{corollary}

We now need to prove that $\C[\mathcal{V}_\mcL]^{\GL_{n,\omega}}$ is Cohen-Macaulay. However, we know that this ring is isomorphic to $\mcP_n[\mcL]^{G_n}$ which is Cohen-Macualay by the following theorem.
\begin{theorem}[\cite{hochster1974rings}]
 For a reductive group $G$ acting rationally on a vector space $V$, $\C[V]^G$ is Cohen-Macaulay.
\end{theorem}

When studying orbit closures, the following theorem is a powerful tool when dealing with reductive groups.

\begin{theorem}[The Hilbert-Mumford Criterion \cite{kempf1978instability}]\label{thm:hbcriterion}
 For a linearly reductive group $G$ acting on a variety $V$, if $\overline{G.w}\setminus G.w\ne\emptyset$, then there exists a $v\in\overline{G.w}\setminus G.w$ and a 1-parameter subgroup (or cocharacter) $\lambda:k^\times\to G$ (where $\lambda$ is a homomorphism of algebraic groups), such that $\lim_{t\to 0}{\lambda(t).w}=v$. 
\end{theorem}

A cocharachter of $\GL_{n,\omega}$ is a product of cocharachters $\lambda(t)=\times_{i=1}^n{\lambda_i(t)}$ of $\GL(\C^\omega)$. If $\lambda(t).(\ott_{i=1}^n{M_i})\to 0$, it is easy to see that for some $i$, $\lambda_i(t).M_i\to 0$. This means that $M_i$ is in the null cone of $\GL(\C^\omega)\actson\End(\C^\omega)$ by conjugation. It is well known that this implies that $M_i$ is nilpotent (cf. \cite{brion2008representations}). So the null cone of $\mathcal{V}$ are those matrices $\ott_{i=1}^n{M_i}$ where at least one $M_i$ is nilpotent. 

\begin{defn}
 Given a tuple $(d_1,\dots,d_n)$, let $d=\tn{lcm}\{d_1,\dots,d_n\}$. Then define $$\Tr_{d_1,\dots,d_n}(\ott_{i=1}^n{M_i}):=\prod_{i=1}^n{\Tr(M^{d_i})^{d/d_i}}.$$
\end{defn}

\begin{theorem}\label{thm:segrebound}
 The ring $\C[\mathcal{V_\mcL}]^{\GL_{n,\omega}}$ is generated by polynomials of degree at most $$\tn{max}(2,\frac{3}{8}(\omega^n)\tn{lcm}\{w|\;w\in w(\mcL)\}^2).$$
\end{theorem}
\begin{proof}
 We first prove that the vanishing locus of $\Tr_{d_1,\dots,d_n}$ is precisely that null cone. It is clear that every element of the null cone is in the vanishing locus of these polynomials. Now suppose $M=\ott_{i=1}^n{M_i}$ is not in the null cone. Then for each $i$, there is a $k_i\in[\omega]$ such that $\Tr(M_i^{k_i})\ne 0$ as $M_i$ is not nilpotent. Then $\Tr_{k_1,\dots,k_n}(M)\ne 0$. So the vanishing locus of these polynomials is precisely the null cone. But we note that by restricting these functions to the variety $\mathcal{V}_\mcL$, the $\Tr_{d_1,\dots,d_n}$ is identically zero unless all $d_i\in w(\mcL)$ and so can be excluded from the set of polynomials whose vanishing locus defines the null cone. Thus the degree of these polynomials is bounded by $\tn{lcm}\{w|\;w\in w(\mcL)\}$. 
 The dimension of $\C[\mathcal{V}_\mcL]^{\GL_{n,\omega}}$ is upper bounded by the dimension of $\C[\End(\C^w)^{\ot n}]=\omega^n$, completing the proof.
 \end{proof}

\begin{corollary}\label{cor:segrebound}
 If $\mcL$ is a set of rectangular polynomials all of unit height then $$\beta_n(\mcL)\le\tn{max}(2,\frac{3}{8}(\omega^n)\tn{lcm}\{w|\;w\in w(\mcL)\}^2)$$ where $\omega$ is the maximum width of any polyomino in $\mcL$.
\end{corollary}
\begin{proof}
 This follows from Theorem \ref{thm:segrebound} and the observation that $\mcP_n[\mcL]^{G_n}$ is isomorphic to a quotient of $\C[\mathcal{V}_\mcL]^{\GL_{n,\omega}}$.
\end{proof}

\section{Conclusion}

We end this paper with several questions. We introduced the notion of $n$-representable sets $\mcL$, and found some families of examples. However, we would be interested in knowing if other examples exist, possibly by choosing rectangles in $\mcL$ more carefully than simply restricting their heights. 

Of course, $n$-representability seems a much stronger condition than being strongly $n$-Noetherian. In our proofs, we only allowed horizontal translations. If all translations and rotations are allowed, it is tempting to believe that every set of rectangular polyominoes is $n$-Noetherian. However, it doesn't seem terribly unlikely that this is not the case as well. So a question of interest is if there exists a set of rectangular polyominoes $\mcL$ that are not (strongly or otherwise) $n$-Noetherian.

We restricted ourselves to rectangular polyominoes, but what if other relatively nice polyominoes are allowed. Does there exists a set of polyominoes $\mcL$ that includes a hook, for example, that are (strongly) $n$-Noetherian? We also looked only at vertical faults, but another natural variant is to ask whether a tiling of $R_{n\times m}$ can be rearranged to give a vertical \emph{or horizontal} fault. 

Lastly, while we were able to give upper bounds, they may be far from optimal. And we did not give any lower bounds. The least common multiple of the widths of polyominoes on $\mcL$ is an easy lower bound to see. However, it it s not clear how sharp this lower bound is in general.
\subsection*{Acknowledgments} The author would like to thank Viresh Patel, Guus Regts, and  Lex Schrijver for their helpful discussions. The research leading to these results has received funding from the European Research Council under the European Union's Seventh Framework Programme (FP7/2007-2013) / ERC grant agreement No 339109.
\bibliographystyle{plain}
\bibliography{bibfile}

 \end{document}